\documentclass[11pt,a4paper,twoside]{article}
\usepackage{graphics}
\usepackage{amssymb}
\usepackage{mathrsfs}
\usepackage{makeidx}
\usepackage{color}
\usepackage{graphicx}
\usepackage{subfigure}
\usepackage{multicol}
\usepackage{multirow}
\usepackage{float}
\usepackage{booktabs}
\usepackage{epsfig}
\usepackage{epstopdf}
\usepackage[colorlinks,
            linkcolor=red,
            anchorcolor=blue,
            citecolor=blue
            ]{hyperref}
\usepackage{caption}
\usepackage{algorithm,algorithmic}
\usepackage{geometry}

\usepackage{latexsym, cite, bm, amsthm, amsmath}
\usepackage{enumerate}
\usepackage{marginnote}
\usepackage{array}
\usepackage{booktabs}

\def \[{\begin{equation}}
\def \]{\end{equation}}

\newtheorem{Theorem}{Theorem}[section]
\newtheorem{Corollary}{Corollary}[section]
\newtheorem{Lemma}{Lemma}[section]

\newtheorem{Algorithm}{Algorithm}[section]
\newtheorem{Remark}{Remark}[section]

\newtheorem{Example}{Example}[section]
\numberwithin{equation}{section}
\title{\bf On finite termination of the generalized Newton method for solving absolute value equations}
\usepackage[marginal]{footmisc}
\usepackage{authblk}
\author[a]{Jia Tang\thanks{Supported partially by the  the Natural Science Foundation of Fujian Province (Grand No. 2020J01166). Email address: tang\_jia@126.com.}}
\author[a]{Wenli Zheng\thanks{Email address: 735349940@qq.com.}}
\author[a]{Cairong Chen\thanks{Corresponding author. Supported partially by the National Natural Science Foundation of China (Grant No. 11901024) and  the Natural Science Foundation of Fujian Province (Grand No. 2021J01661). Email address: cairongchen@fjnu.edu.cn.}}
\author[b]{Dongmei Yu\thanks{Supported partially by the Ministry of Education in China of Humanities and Social Science Project (Grand No. 21YJCZH204), the Natural Science Foundation of Liaoning Province (Grand Nos. 2020-MS-301, LJ2020ZD002) and the Youth Talent Entrustment Project of Liaoning Provincial Federation Social Science Circles (Grand No. 2022lslwtkt-069). Email: {yudongmei1113@163.com}.}}
\author[c]{Deren Han\thanks{Supported partially by the National Natural Science Foundation of China (Grant Nos. 12131004 and 11625105). Email: {handr@buaa.edu.cn}.}}
\affil[a]{School of Mathematics and Statistics, FJKLMAA and Center for Applied Mathematics of Fujian Province, Fujian Normal University, Fuzhou, 350007, P.R. China}
\affil[b]{Institute for Optimization and Decision Analytics, Liaoning Technical University, Fuxin, 123000, P.R. China}
\affil[c]{LMIB of the Ministry of Education, School of Mathematical Sciences, Beihang University, Beijing 100191, P.R. China}

\begin{document}
\maketitle
\begin{quote}
{\bf Abstract:}  Motivated by the framework constructed by Brugnano and Casulli $[$SIAM J. Sci. Comput. 30: 463--472, 2008$]$, we analyze the finite termination property of the generalized Netwon method (GNM) for solving the absolute value equation (AVE). More precisely, for some special matrices, GNM is terminated in at most $2n + 2$ iterations. A new result for the unique solvability and unsolvability of the AVE is obtained. Numerical experiments are given to demonstrate the theoretical analysis.

{\small

}

\end{quote}

\section{Introduction}
In this paper, we consider the following absolute value equation~(AVE):
\begin{equation}\label{eq:ave}
    Ax-|x|-b=0,
\end{equation}
where $A\in \mathbb{R}^{n\times n},\, b\in \mathbb{R}^n$ and $|x|=(|x_1|,\cdots,|x_n|)^\top$ represents the componentwise absolute value of the unknown vector $x$. Recently, AVE~\eqref{eq:ave} has attracted more and more attention in the optimization community since its relevance to many mathematical programming problems, including the linear complementarity problem, the bimatrix game and others (see, for instance, \cite{mang2007,mame2006,prok2009} and the references therein).

In general, solving AVE~\eqref{eq:ave} is NP-hard \cite{mang2007}. In addition, if AVE~\eqref{eq:ave} is solvable, checking whether AVE~\eqref{eq:ave} has a unique solution or multiple solutions is NP-complete~\cite{prok2009}. Nevertheless, some necessary or sufficient conditions for the unique solvability of AVE~\eqref{eq:ave} have been constructed, see, for example, \cite{mezzadri2020,mame2006,rohn2014,zhang2009,wu2018,wu2021,hladik2022,wu2016} and the references therein. In particular, one of the known sufficient conditions for the unique solvability of AVE~\eqref{eq:ave} is described in Lemma~\ref{lem:sc}.

\begin{Lemma}[\cite{mame2006}]\label{lem:sc}
AVE~\eqref{eq:ave} is uniquely solvable for any $b\in \mathbb{R}^n$ if $\|A^{-1}\|<1$.
\end{Lemma}

In the past few years, many numerical methods have been proposed to solve the unique solution of AVE~\eqref{eq:ave}. For example, the generalized Newton method (GNM)  \cite{mang2009} and its extensions \cite{caqz2011,crfp2016,wacc2019}, the SOR-like iterative method \cite{kema2017,guwl2019}, the concave minimization method \cite{zamani2021,mang2007,mang2007b}, the Levenberg-Marquardt method \cite{iqia2015}, the generalized Gauss-Seidel iterative method \cite{edhs2017}, the exact and inexact Douglas-Rachford splitting methods \cite{chen2022} and others, see, e.g.,  \cite{raam2019,guhl2017,yuch2020,ke2020,sayc2018, maer2018,cyyh2021,abhm2018} and the references therein. Most of the above mentioned methods are proved to be convergent under the condition that $\|A^{-1}\|<1$ (In this situation, it follows from Lemma~\ref{lem:sc} that AVE~\eqref{eq:ave} is uniquely solvable). Among them,  GNM often obtains the exact solution of AVE~\eqref{eq:ave} in just a few iterations, which makes it a competitive method.  Theoretically, however, GNM can only be applied to the case with more restrictions. Concretely, Mangasarian in \cite{mang2009} showed that the sequence generated by GNM is well defined and linearly convergent if $\|A^{-1}\|<\frac{1}{4}$. Then it was proved that the condition  $\|A^{-1}\|<\frac{1}{4}$ can be relaxed to $\|A^{-1}\|<\frac{1}{3}$ by the authors in \cite{lian2018,crfp2016}. Numerically, it turns out that GNM works under the less stringent condition that $\|A^{-1}\|> 0$ \cite{chen2022,mang2009}. Thus, it would be very useful to establish convergence of GNM under this assumption. This paper is devoting to solve this problem to some extend. Indeed, under mild conditions, we conclude that GNM is terminated in at most $2n+2$ iterations. We should mention that Mangasarian in \cite{mang2009} also presented the finite termination property of GNM (see Lemma~\ref{lem:termi} below), but he did not give the upper bound. Our work here is inspired by \cite{Brugnano2008}.

The rest of this paper is organized as follows. In the next section, some preliminaries are given. The finite termination property of GNM is further discussed in Section~\ref{sec:main}. Numerical experiments are reported in Section~\ref{sec:numerical}. Conclusions are made in Section~\ref{sec:conclusion}.

\textbf{Notation.} We use $\mathbb{R}^{n\times n}$ to denote the set of all $n \times n$ real matrices and $\mathbb{R}^{n}= \mathbb{R}^{n\times 1}$. $I$ is the identity matrix with suitable dimension. $\|A\|$ denotes the spectral norm of the matrix $A$. For a matrix $X$, $\mathcal{N}(X)$ is the null space of $X$. Let $\rho(A)$ denote the spectral radius of the square matrix $A$. Throughout this paper, $\mathcal{D}(x) \doteq \textbf{diag}(\textbf{sign}(x))$ with $\textbf{sign}(x)$ denoting a vector with components equal to $-1,\, 0$, or $1$, respectively, depending on whether the corresponding element in the vector $x$ is negative, zero, or positive; and for $x \in \mathbb{R}^{n}$, $\textbf{diag}(x)$ represents a diagonal matrix with $x_i$ as its diagonal entries for every $i=1,2,\cdots,n$. For a vector $v\in \mathbb{R}^n$, $\textbf{span}(v)$ denotes the linear space spanned by $v$. For $X\in\mathbb{R}^{m\times n}$, $X_{(i,j)}$ refers to its $(i,j)$th entry, $|X|$ is in
$\mathbb{R}^{m\times n}$ with its $(i,j)$th entry $|X_{(i,j)}|$.
Inequality $X\ge Y$ means $X_{(i,j)}\ge Y_{(i,j)}$ for all $(i,j)$. In particular, $X\ge
0$ means that $X$ is a nonnegative matrix.

\section{Preliminaries}
In this section, we briefly introduce some preliminaries, which lay the foundation for our later arguments.

A matrix $A\in\mathbb{R}^{n\times n}$ is called a {\em $Z$-matrix\/} if $A_{(i,j)}\le 0$ for all $i\ne j$. A $Z$-matrix $A$ is called an $M$-matrix if $A^{-1}\ge 0$ \cite{Berman1994}. The  following property can be found in \cite{Berman1994}.

\begin{Lemma}\label{thm:MMtx-2}
Let $A\in \mathbb{R}^{n\times n}$ be an $M$-matrix. Let $B\in \mathbb{R}^{n\times n}$ be a $Z$-matrix. If $B\ge A$, then $B$ is also an $M$-matrix.
\end{Lemma}

We
assume, throughout the rest of this paper, that either
\begin{subequations}\label{eq:condMM}
\begin{equation}\label{eq:condMM-1}
\mbox{
$A - I$ is an $M$-matrix}
\end{equation}
or
\begin{equation}\label{eq:condMM-2}
\framebox{
\parbox{8.5cm}{
$\mathcal{N}\left(A^\top-I\right) = \textbf{span}(v)$ with $v>0$ and $A - I + D$ is an $M$-matrix for all diagonal matrices $D=\textbf{diag}(d)$ such that $d\ge 0$ and $d\neq 0$.
}
}
\end{equation}
\end{subequations}

\begin{Remark}
The assumption \eqref{eq:condMM} is inspired by \cite{Brugnano2008}. In \cite{Brugnano2008}, the piecewise linear system
$$
\max\{0,x\} + Tx = c
$$
is considered, which is equivalent to
\begin{equation}\label{eq:nave}
(I + 2T)x + |x| = 2c.
\end{equation}
Obviously, AVE~\eqref{eq:nave} can be rewritten as
\begin{equation}\label{eq:ave2}
   Ax + |x| =b
\end{equation}
with $A = I + 2T$ and $b = 2c$. Note that AVE~\eqref{eq:ave2} can be reformulated as
$$A(-x) - |-x| = -b,$$
which is $Ay - |y| = -b$ with $y = -x$. Similarly, AVE~\eqref{eq:ave} can be reformulated as
$$
Ay + |y| = -b
$$
with $y = -x$. In this sense, we do not distinguish AVE~\eqref{eq:ave} and AVE~\eqref{eq:ave2} and we focus on AVE~\eqref{eq:ave} in this paper.
\end{Remark}

For the convenience of reference, we formally present GNM for AVE~\eqref{eq:ave} in Algorithm~\ref{alg}.
\begin{Algorithm}[\cite{mang2009}]\label{alg}
Assume that $x^0\in\mathbb{R}^n$ is an arbitrary initial guess. For $k=0,1,2,\cdots$ until the iterative sequence $\{x^k\}_{k=0}^\infty$ is convergent, computing $x^{k+1}$ by
\begin{equation}\label{eq:newton}
      x^{k+1} = \left[A - \mathcal{D}(x^k)\right]^\mathsf{-1}b.
\end{equation}
\end{Algorithm}

The following lemma implies that GNM~\eqref{eq:newton} is well defined when $\|A^\mathsf{-1}\| < 1$.

\begin{Lemma}[\cite{caqz2011}]\label{lem:nons}
Suppose that $\|A^\mathsf{-1}\| < 1$ and $D=\textbf{diag}(d)$ with $d_i\in[-1,1]$\,$(i=1,2,\cdots,n)$. Then, $A-D$ is nonsingular.
\end{Lemma}

The following lemma shows the finite termination property of GNM iteration~\eqref{eq:newton}. However, the upper bound of $k$ is unknown.

\begin{Lemma}[\cite{mang2009}]\label{lem:termi}
Let $\|A^\mathsf{-1}\| < 1$. If $\mathcal{D}(x^{k+1}) = \mathcal{D}(x^k)$ for some $k$ for the well defined GNM iteration~\eqref{eq:newton}, then $x^{k+1}$ solves AVE~\eqref{eq:ave}.
\end{Lemma}

\section{Main Results}\label{sec:main}
In this section, based on the assumption~\eqref{eq:condMM}, we will further discuss the finite termination property of GNM iteration~\eqref{eq:newton}.

\begin{Lemma}\label{lem:cond}
Let the matrix $A$ satisfy either \eqref{eq:condMM-1} or \eqref{eq:condMM-2}. If $A$ satisfies \eqref{eq:condMM-2}, assume also that $\mathcal{D}(x^0) \neq I$ and $v^\top b < 0$. Then $A - \mathcal{D}(x^k)$ is an $M$-matrix for $k = 0,1,2,\cdots$, and GNM~\eqref{eq:newton} is well defined.
\end{Lemma}
\begin{proof}
    If $A$ satisfies \eqref{eq:condMM-1}, it follows from Lemma~\ref{thm:MMtx-2} that $A - \mathcal{D}(x^k)$ is an $M$-matrix with $\mathcal{D}(x^k) \in [-I,I]$ and thus GNM~\eqref{eq:newton} is well defined.

    If $A$ satisfies \eqref{eq:condMM-2} and $\mathcal{D}(x^0) \neq I$, then $A - \mathcal{D} (x^0) = A - I + [I-\mathcal{D}(x^0)]$ is an $M$-matrix. Since $v^\top b< 0$, one has
    $$
    v^\top [A - \mathcal{D}(x^0)] x^1 = v^\top \{(A - I) + [I-\mathcal{D}(x^0)]\}x^1 = v^\top [I - \mathcal{D}(x^0)]x^1 = v^\top b < 0,
    $$
    which implies that at least one entry of $x^1$ is strictly negative. Thus $\mathcal{D}(x^1)\neq I$ and $A-\mathcal{D}(x^1) = A - I + [I - \mathcal{D}(x^1)]$ is an $M$-matrix. In the same manner, we can recursively prove that $A-\mathcal{D}(x^k)$ is an  $M$-matrix with $k \ge 2$ and thus GNM~\eqref{eq:newton} is also well defined.
\end{proof}

We are now in the position to present the finite termination property of GNM~\eqref{eq:newton}.

\begin{Theorem}\label{thm:exist}
Let the matrix $A$ satisfy either \eqref{eq:condMM-1} or \eqref{eq:condMM-2}. If $A$ satisfies \eqref{eq:condMM-2}, assume also that $\mathcal{D}(x^0) \neq I$ and $v^\top b< 0$. Then, GNM~\eqref{eq:newton} converges to an exact solution of AVE~\eqref{eq:ave} in at most $2n+2$ iterations.
\end{Theorem}

\begin{proof}
The proof is inspired by~\cite[Theorem~2]{Brugnano2008}. It follows from Lemma~\ref{lem:cond} that GNM~\eqref{eq:newton} is well defined. By iterative scheme~\eqref{eq:newton}, we have
\begin{equation*}
   [A - \mathcal{D}(x^k)]x^{k+1} = [A - \mathcal{D}(x^{k-1})]x^k = b, \quad k = 1,2,\cdots,
\end{equation*}
which implies
\begin{equation}\label{eq:con1}
   [A - \mathcal{D}(x^k)]x^{k+1} = [A - \mathcal{D}(x^k)]x^k + \xi^k,\quad k = 1,2,\cdots,
\end{equation}
where $\xi^k = [\mathcal{D}(x^k) - \mathcal{D}(x^{k-1})]x^k \ge 0$. In fact, if we denote ${d_i}^k$ the $i$th diagonal entry of $\mathcal{D}(x^k)$, one has
  \begin{equation*}
    \begin{cases}
    {x_i}^k > 0 \quad \Leftrightarrow \quad {d_i}^k = 1 \quad \Rightarrow \quad {d_i}^k-{d_i}^{k-1} \ge 0 \quad \Rightarrow \quad ({d_i}^k-{d_i}^{k-1}){x_i}^k \ge 0, \\
    {x_i}^k = 0  \quad \Rightarrow \quad ({d_i}^k-{d_i}^{k-1}){x_i}^k  = 0, \\
    {x_i}^k < 0 \quad \Leftrightarrow \quad {d_i}^k = -1 \quad \Rightarrow \quad {d_i}^k-{d_i}^{k-1} \le 0 \quad \Rightarrow \quad ({d_i}^k-{d_i}^{k-1}){x_i}^k  \ge 0.
    \end{cases}
  \end{equation*}
Since $[A - \mathcal{D}(x^k)]^\mathsf{-1} \ge 0$ and $\xi^k \ge 0$, it follows from~\eqref{eq:con1} that  $x^{k+1}\ge x^k$. Hence, $\mathcal{D}(x^{k+1})\ge \mathcal{D}(x^k)\, (k=1,2,\cdots)$. From the proof of Lemma~\ref{lem:termi}, if $\mathcal{D}(x^{k+1}) = \mathcal{D}(x^k)$ for some $k\ge 0$, then $x^{k+1}$ is an exact solution of AVE~\eqref{eq:ave}. If $\mathcal{D}(x^{k+1}) \ne \mathcal{D}(x^{k})$, namely, $d_i^{k+1} > d_i^k$ for some $i\in \{1,2,\cdots,n\}$ and $k\ge 1$, then $\mathcal{D}(x^{k+1}) \ne \mathcal{D}(x^{k})$ may occur at most $n - m + 1$ times, where
  $$
  m = \sum_{i=1}^n d_i^1.
  $$
In the hypothetical case that $\mathcal{D}(x^{0}) \neq \mathcal{D}(x^{1})$, $m = -n$ and $\sum_{i=1}^n {d_i}^k= k-n-1$ ($k\ge 1$), $\mathcal{D}(x^{k+1}) \ne \mathcal{D}(x^{k})$ would occur $2n+1$ times. Therefore, GNM~\eqref{eq:newton} converges to an exact solution of AVE~\eqref{eq:ave} in at most $2n+2$ iterations.
\end{proof}

As by-products of Theorem~\ref{thm:exist}, we have the following results.

\begin{Theorem}\label{thm:unique}
Let the matrix $A$ satisfy either \eqref{eq:condMM-1} or \eqref{eq:condMM-2}. If $A$ satisfies \eqref{eq:condMM-2}, assume also that $v^\top b< 0$. Then the solution of AVE~\eqref{eq:ave} exists and is unique.
\end{Theorem}

\begin{proof}
The existence of a solution is proved constructively by Theorem~\ref{thm:exist}. In the following, we prove the uniqueness.

For any two vectors $x, y\in\mathbb{R}^n$, one has
\begin{equation*}
     D(x)x - D(y)y = Q(x-y),
\end{equation*}
where $Q=\textbf{diag}(q)\in\mathbb{R}^{n\times n}$ is a diagonal matrix whose diagonal entries $q_i\in [-1,1]\,(i = 1,2,\cdots,n)$. Indeed, we have
   \begin{enumerate}
   \item[(1)] $x_i,y_i \ge 0   \quad \Rightarrow \quad \textbf{sign}(x_i) = 0~\text{or}~1,  \textbf{sign}(y_i) = 0~\text{or}~1 \quad \Rightarrow \quad q_i =1$;

   \item[(2)] $x_i,y_i < 0   \quad \Rightarrow \quad \textbf{sign}(x_i) = \textbf{sign}(y_i) =-1  \quad \Rightarrow \quad q_i =-1$;

   \item[(3)] $x_i\ge 0 > y_i   \quad \Rightarrow \quad \textbf{sign}(x_i)=1~\text{or}~0, \textbf{sign}(y_i) = -1 \quad \Rightarrow \quad -1 \le q_i < 1$;

   \item[(4)] $x_i< 0 \le y_i   \quad \Rightarrow \quad \textbf{sign}(x_i)=-1, \textbf{sign}(y_i) = 1~\text{or}~0 \quad \Rightarrow \quad -1 \le q_i < 1$.
   \end{enumerate}

Now, we assume that $x$ and $y$ are both solutions of AVE~\eqref{eq:ave}, i.e.,
$$
[A - D(x)]x = b, \quad [A - D(y)]y = b.
$$
If $A$ satisfies \eqref{eq:condMM-1}, it is obvious that $A - Q$ is an $M$-matrix.
Therefore,
\begin{equation}\label{eq:unique}
     [A - D(x)]x - [A - D(y)]y = (A - Q)(x - y) = 0,
\end{equation}
which implies $x=y$ according to the nonsingularity of $A - Q$. If $A$ satisfies \eqref{eq:condMM-2} and $v^\top b< 0$, we have
\begin{align*}
v^\top [A - D(x)]x &= v^\top [I - D(x)]x = v^\top b <0,\\
v^\top [A - D(y)]y &= v^\top [I - D(y)]y = v^\top b <0.\\
\end{align*}
Therefore, there is $i$ and $j$ such that $x_i <0$ and $y_j<0$, which implies that at lest one of the diagonal entries of $Q$ is strictly less than $1$. Thus, $A - Q = A-I + (I-Q)$ is an $M$-matrix and it follows
from~\eqref{eq:unique} that $x = y$.
\end{proof}

\begin{Remark}
The unique solvability condition \eqref{eq:condMM-1} of AVE~\eqref{eq:ave} is also studied by Wu and Guo in \cite{wu2016} and we give a new proof in Theorem~\ref{thm:unique}. In addition, Hlad\'{i}k and Moosaei in \cite{hladik2022} show that the assumption \eqref{eq:condMM-1} implies $\rho(|A^{-1}|)<1$, a sufficient condition due to Rohn et al. \cite{rohn2014}. If the condition \eqref{eq:condMM-1} holds, $A$ is an $M$-matrix and $\|A^{-1}\|<1$ implies $\rho(|A^{-1}|) = \rho(A^{-1})<1$. In general, however, $\|A^{-1}\|<1$ does not imply $\rho(|A^{-1}|) <1$ \cite{zhang2009}. Moreover, the condition \eqref{eq:condMM-1} does not imply $\|A^{-1}\|<1$ and vice versa. For example, let \cite{zhang2009}
$$
A = \left(
    \begin{array}{cc}
    1.5 &-3\\
    0& 1.5
\end{array}\right),$$
then $A-I$ is an $M$-matrix but $\|A^{-1}\| \approx 1.6095 >1$. Let
$$
A = \left(
    \begin{array}{cc}
    1.5 &-1.25\\
    0& 1.5
\end{array}\right),$$
we have $\|A^{-1}\| = 1$ and $A-I$ is an $M$-matrix. On the other hand, for the matrix satisfied $\|A^{-1}\|<1$, $A-I$ may not be a $Z$-matrix not to mention satisfying \eqref{eq:condMM-1}. For example \cite{zhang2009},
 $$
A = \left(
    \begin{array}{cc}
    1 &-0.01\\
    0.01& 1
\end{array}\right).$$
However, when $A$ is symmetric and $\|A^{-1}\| \ge 1$, it can be proved that the assumption \eqref{eq:condMM-1} does not hold.

It follows from Lemma~\ref{lem:nons} that the assumption \eqref{eq:condMM-2} does not hold whenever $\|A^{-1}\|<1$. When $\|A^{-1}\|\ge 1$, the assumption \eqref{eq:condMM-2} can hold. For instance (see Examples~\ref{ex4}-\ref{ex5} for more detail), the matrix
$$
A = \left(
  \begin{array}{cc}
   3   & -2  \\
   -2  &   3  \\
   \end{array}
   \right)
$$
satisfies the assumption \eqref{eq:condMM-2} and $\|A^{-1}\| =1$, and the matrix
$$
A = \left(
  \begin{array}{cc}
   3   & -1  \\
   -4  &   3  \\
   \end{array}
   \right)
$$
satisfies the assumption \eqref{eq:condMM-2} and $\|A^{-1}\| \approx 1.1708 > 1$.
\end{Remark}

Furthermore, we have the following results, which are different from the existing results \cite{mame2006,prok2009}.
\begin{Corollary}
Let the matrix $A$ satisfy \eqref{eq:condMM-2}. Then we have the following claims.
\begin{itemize}
  \item [(1)] If $v^\top b = 0$ and $A$ is symmetric, then a solution of AVE~\eqref{eq:ave} exists but is not unique;

  \item [(2)] If $v^\top b > 0$, then AVE~\eqref{eq:ave} has no solutions.
\end{itemize}
\end{Corollary}

\begin{proof}
If $v^\top b =0$ and $A$ is symmetric, then $b$ is in the range of $A - I$. Thus, there exists a vector $u$ such that $(A - I)u = b$. Let $\alpha < \min_i \frac{u_i}{v_i}$ and denote
$$
x(\alpha) = u - \alpha v,
$$
then $x(\alpha) >0$ and
$$
(A - I)x(\alpha) = (A-I)u = b.
$$
Consequently, $x(\alpha)$ is solution of AVE~\eqref{eq:ave}.

If $v^\top b > 0$, let $x$ be a solution of AVE~\eqref{eq:ave}, then we have
$$
  v^\top [A - D(x)]x = v^\top [A - I + I - D(x)]x = v^\top [I - D(x)]x = v^\top b >0,
$$
which is impossible since $x - |x|\le 0$ and $v>0$.
\end{proof}

\section{Numerical Experiments}\label{sec:numerical}
In this section, we will present five numerical examples to illustrate the theoretical analysis presented in this paper. The purpose here is to show that GNM is applicable under the  conditions in Theorem~\ref{thm:exist}, no matter $\|A^\mathsf{-1}\| < 1$, $\|A^\mathsf{-1}\| = 1$ or $\|A^\mathsf{-1}\| > 1$. In our implementations, GNM is terminated if $RES(x^k) = \left\|Ax^k - |x^k| - b \right\| \le 10^{-7}$. For the first three examples, we choose $x^0 = (1,1,\cdots,1)^\top$ as the initial point. For the last two examples, we choose $x^0 = (1,-1)^\top$. All experiments are implemented in MATLAB R2018b.

The first example is inspired by \cite{Brugnano2008}, in which we have $\|A^{-1}\| <1$ and the assumption \eqref{eq:condMM-1} is satisfied.

\begin{Example}\label{ex1}
Consider AVE~\eqref{eq:ave} with
  $$
    A = \left(
    \begin{array}{cccccc}
      7   &   -2   &    0   & \cdots &    0   &   0    \\
     -2   &    7   &   -2   & \cdots &    0   &   0    \\
      0   &   -2   &    7   & \cdots &    0   &   0    \\
   \vdots & \vdots &        & \ddots & \vdots & \vdots \\
      0   &    0   & \cdots & \cdots &    7   &  -2    \\
      0   &    0   & \cdots & \cdots &   -2   &   7    \\
    \end{array}
    \right)\in\mathbb{R}^{n\times n}\quad \text{and} \quad b = Ax^* - |x^*|,
  $$
where ${x_i}^* = e^{6\frac{i-1}{n-1}-5}-1\,(i = 1,2,\cdots,n)$.

Numerical results are reported in Table~\ref{tab1} (where `IT' denotes the number of iterations). It follows from Table~\ref{tab1} that GNM converges to the exact solution in at most three iterations for different $n$.

\setlength{\tabcolsep}{6.0pt}
\begin{table}[htb]
\centering
\caption{Numerical results for Example \ref{ex1}.}
 \begin{tabular}{ccc}\hline
	\toprule
	    $n$       &    IT     &    RES             \\
	\midrule
	  $2000$     &   $3$    &  $2.6634e-14$   \\
	  $4000$     &   $3$    &  $3.6610e-14$      \\
	  $6000$     &   $2$    &  $4.5385e-14$      \\
      $8000$     &   $3$    &  $5.1692e-14$      \\
      $10000$    &   $2$    &  $5.7936e-14$     \\
	\bottomrule
	\end{tabular}
	\label{tab1}
\end{table}

\end{Example}

When the matrix $A$ satisfies the assumption \eqref{eq:condMM-1} and $\|A^\mathsf{-1}\| \ge  1$, we have the following two examples.

\begin{Example}\label{ex2}
Consider AVE~\eqref{eq:ave} with
 $$
  A = \left(
  \begin{array}{cc}
   1.5  &-1.25  \\
   0    &  1.5  \\
   \end{array}
   \right) \quad \text{and} \quad
   b = \left(
        \begin{array}{c}
          4  \\
          16 \\
       \end{array}
       \right).
  $$
 Note that $\|A^\mathsf{-1}\|=1$. Moreover, the unique solution of the AVE~\eqref{eq:ave} is $x^* = (88,32)^\top$ and GNM converges to the exact solution at the first iteration.
\end{Example}

\begin{Example}\label{ex3}
Consider AVE~\eqref{eq:ave} with
  $$
   A = \left(
   \begin{array}{cc}
    1.5  &-3  \\
    0    &  1.5  \\
   \end{array}
   \right) \quad \text{and} \quad
    b = \left(
        \begin{array}{c}
          -2  \\
          -3 \\
       \end{array}
       \right).
  $$
Note that $\|A^\mathsf{-1}\|\approx 1.6095 >1$. Moreover, the unique solution of the AVE~\eqref{eq:ave} is $x^* = (-2.24,-1.2)^\top$. The sequence generated by GNM converges to the $x^*$ at the second iteration.
\end{Example}

Before concluding this section we give two examples with the matrix $A$ satisfying the assumption~\eqref{eq:condMM-2} and $v^\top b <0$.

\begin{Example}\label{ex4}
Consider AVE~\eqref{eq:ave} with
  $$
  A = \left(
  \begin{array}{cc}
   3   &-2  \\
   -2  &   3  \\
   \end{array}
   \right) \quad \text{and} \quad
   b = \left(
        \begin{array}{c}
          -4  \\
          -16 \\
       \end{array}
       \right).
  $$
For this example, we have  $\|A^\mathsf{-1}\| = 1$, $v=(1,1)^\top$, and $v^\top b = -20<0$. GNM converges to the exact solution $x^* = (-4,-6)^\top$ at the second iteration.
\end{Example}

\begin{Example}\label{ex5}
Consider AVE~\eqref{eq:ave} with
 $$
  A = \left(
  \begin{array}{cc}
    3   &-1  \\
    -4  &   3  \\
   \end{array}
   \right) \quad \text{and} \quad
   b = \left(
        \begin{array}{c}
          -5  \\
          -4 \\
       \end{array}
       \right).
  $$
For this example, $\|A^\mathsf{-1}\| \approx 1.1708 >1$, $v=(2,1)^\top$, and $v^\top b = -14<0$. The sequence generated by GNM converges to the exact solution $x^* = (-2,-3)^\top$ at the second iteration.
\end{Example}

In conclusion, under the conditions in Theorem~\ref{thm:exist}, the above mentioned examples show that GNM converges to the unique solution of AVE~\eqref{eq:ave} within $2n +2$ iterations. In fact, it often only takes a few iterations.

\section{Conclusions}\label{sec:conclusion}
For two special classes of matrices, the finite termination property of the generalized Newton method (GNM) for the absolute value equation is further considered in this paper. Under the conditions in Theorem~\ref{thm:exist}, GNM is finitely convergent for $\|A^\mathsf{-1}\|>0$, which is weaker than $ \|A^\mathsf{-1}\|< \frac{1}{3}$ (commonly used in the literature). In addition, a new proof of an existing unique solvability condition is given and a new unique solvability condition and a new unsolvability condition are developed. Numerical results demonstrate our claims.

\end{document}